\documentclass[12pt,a4paper,reqno]{amsart}
\usepackage{amssymb,amsmath}
\usepackage{amsfonts,amsbsy,bm}
\usepackage{latexsym}
\usepackage{exscale}
\usepackage{mathrsfs}
\usepackage{graphicx}
\makeatletter
\newcommand*\bigcdot{\mathpalette\bigcdot@{.5}}
\newcommand*\bigcdot@[2]{\mathbin{\vcenter{\hbox{\scalebox{#2}{$\m@th#1\bullet$}}}}}
\makeatother
\usepackage{color}

\newcommand{\R}{\mathbb R}
\newcommand{\Z}{\mathbb Z}

\renewcommand{\l}{\lambda}

\newcommand{\B}{\mathcal B}

\newcommand{\w}{\omega}
\newcommand{\p}{\varphi}

\newcommand {\SR} {{\mathbb R}}

\newcommand {\ga} {{\gamma}}

\numberwithin{equation}{section}
\newtheorem{theorem}{Theorem}[section]

\newtheorem{Remark}[theorem]{Remark}
\newtheorem{proposition}[theorem]{Proposition}

\newtheorem{example}[theorem]{Example}

\newcommand{\Ba}[1]{\begin{array}{#1}}
\newcommand{\Ea}{\end{array}}
\newcommand{\Be}{\begin{equation}}
\newcommand{\Ee}{\end{equation}}
\newcommand{\Bea}{\begin{eqnarray}}
\newcommand{\Eea}{\end{eqnarray}}
\newcommand{\Beas}{\begin{eqnarray*}}
\newcommand{\Eeas}{\end{eqnarray*}}
\newcommand{\Benu}{\begin{enumerate}}
\newcommand{\Eenu}{\end{enumerate}}
\newcommand{\Bi}{\begin{itemize}}
\newcommand{\Ei}{\end{itemize}}

\newcommand{\BR}{\begin{Remark} \em}
\newcommand{\ER}{\end{Remark}}
\newcommand{\BE}{\begin{example} \em}
\newcommand{\EE}{\end{example}}

\newcommand {\Ds} {\displaystyle}

\newcounter{reg}
\newcounter{regTO}
\begin{document}

\title[Frames of exponentials and sub--multitiles]{Frames of exponentials and sub--multitiles in LCA groups.}

\author{D. Barbieri}

\address{Davide Barbieri \\
Departamento de Matem\'aticas\\
Universidad Aut\'onoma de Madrid\\
28049, Madrid, Spain} 

\email{davide.barbieri@uam.es}

\author{C. Cabrelli}

\address{Carlos Cabrelli\\
	Departamento de Matem\'atica, FCEyN, Universidad de Buenos Aires and IMAS-UBA-CONICET\\ Argentina}
	 \email{cabrelli@dm.uba.ar}

\author{E. Hern\'andez}

\address{Eugenio Hern\'andez \\
Departamento de Matem\'aticas\\
Universidad Aut\'onoma de Madrid\\
28049, Madrid, Spain} \email{eugenio.hernandez@uam.es}

\author{P. Luthy}

\address{Peter Luthy\\
College of Mount Saint Vincent\\
Bronx, NY, USA} \email{peter.luthy@mountsaintvincent.edu}

\author{U. Molter}

\address{Ursula Molter\\
Departamento de Matem\'atica, FCEyN, Universidad de Buenos Aires and IMAS-UBA-CONICET\\ Argentina} 
\email{umolter@dm.uba.ar}

\author{C. Mosquera}

\address{Carolina Mosquera\\
	Departamento de Matem\'atica, FCEyN, Universidad de Buenos Aires and IMAS-UBA-CONICET\\ Argentina} 
\email{mosquera@dm.uba.ar}

\begin{abstract}
In this note we investigate the existence of frames of exponentials for $L^2(\Omega)$
in the setting of  LCA groups. Our main result shows that sub--multitiling properties   of $\Omega \subset \widehat {G}$
 with respect to a uniform lattice $\Gamma$ of $\widehat {G}$ guarantee the existence of a frame of exponentials
 with frequencies in a finite number of translates of the annihilator of $\Gamma$. We also prove the converse of this result and provide  conditions for the existence of these frames. These conditions extend  recent results on Riesz bases of exponentials and multitilings to frames.
\end{abstract}

\date{\today}
\subjclass[2010]{42C15, 06D22}

\keywords{Frames of exponentials, submultitiles, LCA groups. }

\maketitle

\section{Introduction and main result}\label{secIntroduc}

We begin by stating several known results.
\begin{itemize}
\item Let $\Omega$ be a measurable subset of $\SR^d$ with positive, finite measure, let $\Lambda$ be a complete lattice of $\SR^d$ (i.e. $\Lambda= A\Z^d$ for some $d\times d$ invertible matrix $A$ with real entries), and denote by $\Gamma$ the annihilator of $\Lambda.$ Recall that $\Gamma = \{\gamma \in \SR^d: e^{2\pi i\langle \lambda , \gamma \rangle}=1, \forall \ \lambda\in \Lambda \}.$ In 1974, B. Fuglede (\cite{Fue74}, Section 6) proved that $\{e^{2\pi i \langle \lambda , {\,\bigcdot\,} \rangle}: \lambda \in \Lambda\}$ is an orthogonal basis for $L^2(\Omega)$ if and only if $(\Omega, \Gamma)$ is a \textbf{tiling pair} for $\SR^d$, that is $\Ds \sum_{\gamma \in \Gamma} \chi_\Omega (x + \gamma)= 1$, a. e. $x \in \SR^d.$

\

\item The result of B. Fuglede just stated also holds in the setting of locally compact abelian (LCA)  groups. Let $G$ be a second countable LCA group, and let $\Lambda$ be a uniform lattice in $G$ (i.e. $\Lambda$ is a  discrete and co-compact subgroup of $G$). Denote by $\widehat G$ the dual group of $G.$ 
For a character $\omega \in \widehat G$ we use the notation $e_{g} (\omega) = \omega(g),$ for $g\in G.$
Let $\Gamma$ be  the annihilator  of $\Lambda.$ (i.e. $\Gamma = \{\gamma \in \widehat G: e_{\lambda} (\gamma)=1 \text { for all } \lambda \in \Lambda\}$).
The dual group $\widehat G$ of $G$ is also a second countable LCA group, and $\Gamma$ is also a uniform lattice.  Let $\Omega$ be a measurable subset of $\widehat G$ with positive and finite measure. In 1987, S. Pedersen (\cite{Pedersen}, Theorem 3.6) proved that $\{e_\lambda : \lambda \in \Lambda \}$ is an orthogonal basis for $L^2(\Omega)$ if and only if $(\Omega, \Gamma)$ is a tiling pair for $\widehat G$, that is $ \sum_{\gamma \in \Gamma} \chi_\Omega (\omega + \gamma)= 1$, a. e. $\omega \in \widehat G.$

\


\item Recent results in this area focused on \textit{multitiling pairs}. Let $\Omega$ be a bounded, measurable subset of $\R^d$, and let $\Gamma$ be a lattice of $\SR^d$. If there exists a positive integer $\ell$ such that
\[
\sum_{\gamma \in \Gamma} \chi_\Omega (x + \gamma) = \ell\,, \quad \mbox{a.e}\ x\in \SR^d\,,
\]
we will say that $(\Omega , \Gamma)$ is a \textbf{multitiling pair}, or an {\bf $\ell$-tiling pair} for $\SR^d$.
For a lattice $\Lambda \subset \SR^d$ and $a_1, \dots, a_\ell \in \SR^d$, let
\[
E_\Lambda (a_1, \dots , a_\ell) := \{e^{2\pi i\langle a_j+\lambda\,,\, \bigcdot \,\rangle}: j=1, \dots, \ell; \lambda \in \Lambda \}\,.
\]

S. Gresptad and N. Lev (\cite{GL14}, Theorem 1) proved in 2014 that if $\Gamma$ is the annihilator of $\Lambda$, $\Omega$ is a bounded, measurable subset of $\SR^d$ whose boundary has measure zero, and $(\Omega, \Gamma)$ is an $\ell$-tiling pair for $\SR^d$, there exist $a_1, \dots, a_\ell \in \SR^d$ such that $E_\Lambda (a_1, \dots , a_\ell)$ is a \textit{Riesz basis} for $L^2(\Omega).$ The proof of this result in \cite{GL14} uses Meyer's quasicrystals. In 2015  M. Kolountzakis (\cite{Kol15}, Theorem 1) found a simpler and shorter proof without the assumption on the boundary of $\Omega$.
\

For the reader's convenience we recall that a countable collection of elements $\Phi = \{\phi_j: j \in J\}$ of a Hilbert space $\mathbb H$ is a \textbf{Riesz basis} for $\mathbb H$ if it is the image of an orthonormal basis of $\mathbb H$ under a bounded, invertible operator $T \in \mathcal L(\mathbb H).$ Riesz bases provide stable representations of elements of $\mathbb H$.

\

\item This result has been extended to second countable LCA groups by E. Agora, J. Antezana, and C. Cabrelli (\cite{AAC15}, Theorem 4.1). Moreover, they prove the converse (\cite{AAC15}, Theorem 4.4): with the same notation as in the second item of this section, given a relatively compact subset $\Omega$ of $\widehat G$, if $L^2(\Omega)$ admits a Riesz basis of the form
$$
E_\Lambda (a_1, \dots , a_\ell) := \{e_{a_j+\lambda}: j=1,2, \dots, \ell; \lambda \in \Lambda\}
$$
for some $a_1, \dots, a_\ell \in G,$ then $(\Omega, \Gamma)$ is an $\ell$-tiling pair for $\widehat G.$

\end{itemize}

The purpose of this note is to investigate the situation when $(\Omega, \Gamma)$ is a {\it sub-multitiling pair} for $\widehat G.$ Let $\Omega$ be a measurable set in $\widehat G$ with positive and finite Haar measure. For $\Gamma$ a lattice in $\widehat G$ and $\omega \in \widehat G$ define
$$
F_{\Omega, \Gamma}(\omega):= \sum_{\gamma\in \Gamma} \chi_\Omega (\omega + \gamma)\,.
$$ 
If there exists a positive integer $\ell$ such that 
\Be \label{supremum}
\mbox{ess\,sup}_{\omega\in \widehat G} F_{\Omega, \Gamma}(\omega)= \ell\,,
\Ee
we will say that $(\Omega, \Gamma)$ is a {\bf sub-multitiling pair} or an $\ell$-{\bf subtiling pair}.

Denote by $Q_\Gamma$ a fundamental domain of the lattice $\Gamma$ in $\widehat G,$ i.e. it is a Borel measurable section of the quotient group $\widehat G/\Gamma.$ (Its existence is guaranteed by Theorem 1 in \cite{FG68}). Since $F_{\Omega, \Gamma}(\omega)$ is a $\Gamma$-periodic function, it is enough to compute the ess\,sup in (\ref{supremum}) over a fundamental domain $Q_\Gamma$. Observe that $(\Omega, \Gamma)$ is an $\ell$-tiling pair for $\widehat G$ if $F_{\Omega, \Gamma}(\omega)= \ell$ for a. e. $\omega \in Q_\Gamma.$

\

Another structure that allows for stable representations, besides orthonormal and Riesz bases, is that of a \textit{frame}. A collection of elements $\Phi = \{\phi_j: j \in J\}$ of a Hilbert space $\mathbb H$ is a {\bf frame} for $\mathbb H$ if it is the image of an orthonormal basis of $\mathbb H$ under a bounded, surjective operator $T \in \mathcal L(\mathbb H)$ or, equivalently, if there exist $0 < A \leq B < \infty$ such that
$$
A\|f\|^2 \leq \sum_{j\in J} |\langle f , \phi_j\rangle|^2 \leq B \|f\|^2, \quad \mbox{for all}\, f\in \mathbb H.
$$
(See \cite{Young}, Chapter 4, Section 7.) The numbers $A$ and $B$ are called \textbf{frame bounds} of $\Phi.$

\

In this note we prove the following relationship between frames of exponentials in LCA groups and $\ell$-subtiling pairs.

\begin{theorem} \label{Theorem1}
	Let G be a second countable LCA group and let $\Lambda$ be a uniform lattice of $G$.  Let $\widehat G$ be the dual group of $G$, and let $\Gamma$ be the annihilator of $\Lambda.$ Let $\Omega \subset \widehat G$ be a measurable set of positive, finite measure, and let $\ell$ be a positive integer.
	\begin{enumerate}
		\item If for some $a_1, \dots, a_\ell \in G$ the collection $E_\Lambda (a_1, \dots, a_\ell)$ is a frame of $L^2(\Omega),$ then $(\Omega, \Gamma)$ must be an $m$-subtiling pair of $\widehat G$ for some positive integer $m\leq \ell.$	
		
		\item If  $\Omega\subseteq\widehat G$ is a measurable, {\bf bounded} set and $(\Omega, \Gamma)$ is an  $\ell$-subtiling pair of $\widehat G$,
		then there
		exist $a_1,\dots, a_{\ell}\in G$ such that $E_\Lambda(a_1,\dots,a_\ell)$ is a frame of $L^2(\Omega).$	
	\end{enumerate}
\end{theorem}

\begin{Remark}
Recall that any locally compact and second countable group is metrizable, and its metric can be chosen to be invariant under the group action (see \cite{HR79}, Theorem 8.3). Thus, it makes sense to talk about bounded sets in the group $\widehat G.$
\end{Remark}

The proof of Theorem \ref{Theorem1} will be given in Section \ref{ProofTheorem1}. In Section \ref{Sec-optimal} we give other conditions for a set of exponentials of the form $E_\Lambda(a_1,\dots,a_\ell)$ to  be a frame of $L^2(\Omega)$ and provide expressions to compute the frame bounds.

\

{\bf Acknowledgements}. The research of D. Barbieri and E. Hern\'andez is supported by Grants MTM2013-40945-P and MTM2016-76566-P (Ministerio de Econom\'ia y Competitividad, Spain).The research of
C.~Cabrelli, U.~Molter and C.~Mosquera is partially supported by
Grants  PICT 2014-1480 (ANPCyT), PIP 11220150100355 (CONICET) Argentina, and UBACyT 20020130100422BA. P. Luthy was supported by Grant MTM2013-40945-P while this research started at UAM.

\medskip

\section{Proof of Theorem \ref{Theorem1}} \label{ProofTheorem1}

We start with a result that will be used in the proof of part (2) of Theorem \ref{Theorem1}

\begin{proposition} \label{Prop1.1}
If $\Omega$ is a measurable, {\bf bounded} set in $\widehat G$ and $\Gamma$ is a uniform lattice in $\widehat G$ such that $(\Omega , \Gamma)$ is an $\ell$-subtiling pair for $\widehat G$, there exists a {\bf bounded} measurable set $\Delta \subset \widehat G$  such that $\Omega \subset \Delta$ and $(\Delta, \Gamma)$ is an $\ell$-tiling pair for $\widehat G.$
\end{proposition}

\begin{proof}
Let $Q_\Gamma$ be a fundamental domain of $\Gamma$ in $\widehat G.$
Modifying $\Omega$ in a set of measure zero, we can assume that $\sup_{\omega \in Q_\Gamma} F_{\Omega, \Gamma} (\w) = \ell.$
Define
$\widetilde \Gamma = \{\gamma\in \Gamma: \omega + \gamma \in \Omega \ \mbox{for some} \ \omega \in Q_\Gamma  \}.$
Since $\Omega$ is bounded, the set $\widetilde \Gamma$ is finite and, by the definition of $\ell$-subtiling pair, has at least $\ell$ different elements.
	
Set $Q_k = \{\w\in Q_{\Gamma}: F_{\Omega,\Gamma}(\w)=k\}$ for $k=0,1,...,\ell.$ Clearly
$$Q_\Gamma = \bigcup_{k=0}^\ell Q_k\, ,$$ and the union is  disjoint.

Now, for $k=1, \dots, \ell ,$
let $\B_k=\{B\subset \widetilde \Gamma: \#B=k\}.$
For $B \in \B_k$ set
$$Q_k(B) = \{\w \in Q_k : \w + \gamma \in \Omega, \; \ \mbox{for all} \  \gamma \in B\}.$$ 
Since $\Omega$ is measurable, $Q_k$ is measurable and since $Q_k(B) = \bigcap _{\gamma\in B} ((\Omega-\gamma )\cap Q_k),$
then $Q_k(B)$ is also measurable. Observe that the collection $\B_k$ is finite since $\widetilde \Gamma$ is finite. Also
if $B$ and $B'$ are different sets in $\B_k$  then $Q_k(B)\cap Q_k(B')=\emptyset.$ Indeed, if $\omega \in Q_k(B)\cap Q_k(B')$, $\omega + \gamma \in \Omega$ for all $\gamma \in B$ and $\omega + \gamma' \in \Omega$ for all $\gamma' \in B'$. Since $B \neq B'$, there exists $\gamma_1 \in B'\setminus B$. Then, since $\omega \in Q_k,$
$$
k = \sum_{\gamma\in \Gamma}\chi_\Omega(\omega + \gamma) \geq \sum_{\gamma\in B}\chi_\Omega(\omega + \gamma) + \chi_{\Omega}(\omega + \gamma_1) = k + 1\,,
$$
which is a contradiction. Observe that $Q_k = \bigcup_{B\in \B_k} Q_k(B)$, $k=1, \dots, \ell$, and the union is disjoint. Therefore,
\begin{equation} \label{Eq2.1}
\Omega = \bigcup_{k=1}^{\ell} \;\bigcup_{B\in \B_k} \bigcup_{\gamma \in B} Q_k(B) +\gamma\,,
\end{equation}
and the union is disjoint.

For $k=1,\dots,\ell$ and $B\in\B_k$, we extend $B \subseteq \widetilde \Gamma$ to $\widetilde{B}$ by inserting $\ell-k$ distinct elements from $\widetilde \Gamma \setminus B$ into $B$. Let $\widetilde B_0$ be a set of $\ell$ different elements from $\widetilde \Gamma.$ We recall here that  $\# \widetilde\Gamma \geq \ell$ since  $\sup F_{\Omega,\Gamma}=\ell$.
	
Finally we define: 
	
$$\Delta = \Big(\bigcup_{\gamma\in \widetilde B_0} Q_0 + \gamma\Big) \cup \Big(\bigcup_{k=1}^{\ell} \;\bigcup_{B\in \B_k}  \bigcup_{\gamma \in \widetilde{B}} Q_k(B) +\gamma\Big).$$
	
The set $\Delta$ is measurable since it is a finite union of measurable sets. From \eqref{Eq2.1} it is clear that $\Omega \subset \Delta.$ Moreover, if $\w \in Q_k(B),$ for some $B\in \B_k,$ $\w+\gamma \in \Omega$ only when $\gamma \in B.$ Hence, if $\w \in Q_k(B),$ $\w + \widetilde\gamma \in \Delta$ only when $\widetilde \gamma \in \widetilde B.$ Since $\widetilde B$ has precisely $\ell$ elements, if $\w\in Q_k(B)$,
$$
\sum_{\gamma \in \Gamma} \chi_{\Delta}(\w+\gamma) = \sum_{\widetilde \gamma \in \widetilde B}\chi_{\Delta}(\w+ \widetilde\gamma) = \ell\,.
$$
Also, if $\w \in Q_0$
$$
\sum_{\gamma \in \Gamma} \chi_{\Delta}(\w+\gamma) = \sum_{\widetilde \gamma \in \widetilde B_0}\chi_{\Delta}(\w+ \widetilde\gamma) = \ell\,.
$$
Taking into account that $Q_\Gamma = \bigcup_{k=0}^\ell Q_k = Q_0 \cup \Big(
\bigcup_{k=1}^\ell\;\bigcup_{B\in \B_k}  Q_k(B)\Big)\,$ is a disjoint union, we conclude that for $\w\in Q_\Gamma$, $\sum_{\gamma\in \Gamma} \chi_{\Delta}(\w+\gamma) = \ell\,, $ proving that $(\Delta, \Gamma)$ is an $\ell$-tiling pair for $\widehat G.$
\end{proof} 

\begin{Remark}
	The $\ell$-tile found in Proposition \ref{Prop1.1} is not necessarily unique. It depends on the choice of the sets $\widetilde {B}$ and $\widetilde B_0.$
\end{Remark}

For the proof of part (2) of Theorem \ref{Theorem1} we will use the fiberization mapping $\mathcal T : L^2(G) \longrightarrow L^2(Q_\Gamma, \ell^2(\Gamma))$ given by
\begin{equation} \label{Eq2.2}
\mathcal T f(\w) = \{\widehat f(\w + \gamma)\}_{\gamma \in \Gamma} \in \ell^2(\Gamma)\,, \quad \w\in Q_\Gamma.
\end{equation}
The mapping $\mathcal T$ is an isometry and satisfies
\begin{equation} \label{Eq2.3}
\mathcal T(t_\lambda f)(\w) = e_{-\lambda}(\w)\mathcal T f(\w)\,,\quad \lambda\in \Lambda\,, f\in L^2(G)\,, 
\end{equation}
(see Proposition 3.3 and Remark 3.12 in \cite{CP10}), where  $t_{\l}$ denotes the translation by $\l$ that is $t_{\l} f(g) = f (g - \l).$

The next result is Theorem 4.1 of \cite{CP10} adapted to our situation. For $\p_1, \dots, \p_\ell \in L^2(G)$ denote by
$$
S_\Lambda(\p_1, \dots, \p_\ell) :=\overline {\mbox{span}} \{t_\lambda \p_j : \lambda \in \Lambda, j = 1, \dots, \ell\}
$$
the $\Lambda$-invariant space generated by $\{\p_1, \dots, \p_\ell\}\,.$ The measurable range function associated to $ S_\Lambda(\p_1, \dots, \p_\ell) $ is
\begin{equation} \label{Eq2.4}
J(\w) = \overline {\mbox{span}} \{\mathcal T \p_1 (\w), \dots, \mathcal T \p_\ell (\w) \} \subset \ell^2(\Gamma)\,, \quad \w \in Q_\Gamma\,.
\end{equation}
 
\begin{proposition} \label{Prop2.3}  
	Let $\p_1, \dots, \p_\ell \in L^2(G)$ and let $J(\w)$ be the measurable range function associated to $S_\Lambda(\p_1, \dots, \p_\ell)$ as in \eqref{Eq2.4}. Let $0 < A \leq B < \infty\,.$ The following statements are equivalent:
	
	(i) The set $\{t_\lambda \p_j : \lambda \in \Lambda, j = 1, \dots, \ell\}$ is a frame for $ S_\Lambda(\p_1, \dots, \p_\ell) $  with frame bounds $A$ and $B$.
	
	(ii) For almost every $\w \in Q_\Gamma$ the set $\{\mathcal T \p_1 (\w), \dots, \mathcal T \p_\ell (\w) \} \subset \ell^2(\Gamma)\,$ is a frame for $J(\w)$ with frame bounds $A |Q_\Gamma|^{-1}$ and $B |Q_\Gamma|^{-1}\,.$
\end{proposition}

\begin{proof}
Let $f\in S_\Lambda(\p_1, \dots, \p_\ell) $. Use that the fiberization mapping given in \eqref{Eq2.2} is an isometry satisfying \eqref{Eq2.3} to write
\begin{eqnarray}
\sum_{\lambda \in \Lambda} \sum_{j= 1}^\ell |\langle t_\lambda \p_j , f \rangle_{L^2(G)}|^2 & =& \sum_{\lambda \in \Lambda} \sum_{j = 1}^\ell |\langle \mathcal T(t_\lambda \p_j) , \mathcal T f\rangle_{L^2(Q_\Gamma, \ell^2(\Gamma))}|^2 \nonumber \\
&=& \sum_{j= 1}^\ell \sum_{\lambda \in \Lambda}  \Big|\int_{Q_\Gamma} e_{-\lambda}(\w)
\langle \mathcal T(\p_j)(\w) , \mathcal T f(\w)\rangle_{\ell^2(\Gamma)}\,d\w\Big|^2\,. \nonumber
\end{eqnarray}
Since $\{\frac{1}{\sqrt{|Q_\Gamma|}} e_\lambda(\w) : \lambda \in \Lambda  \}$ is an orthonormal basis of $L^2(Q_\Gamma)$ it follows that
$$
\sum_{\lambda \in \Lambda} \sum_{j = 1}^\ell |\langle t_\lambda \p_j , f \rangle_{L^2(G)}|^2 = |Q_\Gamma| \sum_{j=1}^\ell \int_{Q_\Gamma} |\langle \mathcal T\p_j(\w) , \mathcal T f(\w)\rangle_{\ell^2(\Gamma)}|^2 \, d\w\,.
$$
From here, the proof continues as in the proof of Theorem 4.1 in \cite{CP10}. Details are left to the reader.
\end{proof}

\begin{Remark}
	Notice that the factor $|Q_\Gamma|^{-1}$ that appears in $(ii)$ of Proposition \ref{Prop2.3} does not appear in Theorem 4.1 of \cite{CP10}. This is due to the fact that in \cite{CP10} the measure of $Q_\Gamma$ is normalized (see the beginning of Section 3 in \cite{CP10}). Although this fact is not important to prove (2) of Theorem \ref{Theorem1}, it will be crucial in Section 3 to obtain optimal frame bounds of sets of exponentials.
\end{Remark}

\

{\bf Proof of Theorem \ref{Theorem1}}
	
(1) Assume that $E_\Lambda(a_1, \dots, a_\ell)$ is a frame for $L^2(\Omega).$
We define $\varphi \in L^2(G)$ by
	\[
	\widehat{\varphi}:={\chi}_{\Omega}, \quad \text{ and } \quad \varphi_j:= t_{- a_j}\varphi, \quad j=1,\dots, \ell,
	\]
where $t_{a_j}$ denotes the translation by $a_j,$ that is $t_{a_j} \varphi (g)= \varphi (g - a_j).$
	
Since $E_\Lambda(a_1, \dots, a_\ell)$ is a frame of $L^2(\Omega),$ we have that $\{t_{\lambda}\varphi_{j}\colon \lambda\in\Lambda, j=1,\dots, \ell \}$ is a frame of the Paley-Wiener space $PW_{\Omega}:=\{f\in L^2(G)\colon \widehat{f}\in L^2(\Omega)\} = \{f\in L^2(G)\colon  \widehat{f}(\omega)=0, a.e. \  w\in \widehat{G} \setminus \Omega\}.$ This follows from the definition of frame and the fact that for $f\in PW_\Omega$ one has $\|f\|_{L^2(G)} = \|\widehat f\|_{L^2(\Omega)}$ and $\langle f, t_\lambda\varphi_j \rangle_{L^2(G)} = \langle \widehat f , e_{-\lambda +a_j} \rangle_{L^2(\Omega)}.$

In particular,
	\[
	PW_{\Omega} = S_{\Lambda}(\varphi_1,\dots, \varphi_{\ell}):= \overline{\mbox{span}}\{t_{\lambda}\varphi_{j}\colon \lambda\in\Lambda, j=1,\dots, \ell \}.
	\]
That is, $V:=PW_{\Omega}$ is a finitely generated $\Lambda$-invariant space. Denote by $J_V$ the measurable range function of $V$ as given in \eqref{Eq2.4} (see also \cite{CP10}, Section 3, for details).
We now use the fiberization mapping $\mathcal T : L^2(G) \longrightarrow L^2(Q_\Gamma, \ell^2(\Gamma))$ defined in \eqref{Eq2.2}.

By Proposition \ref{Prop2.3},  for a.e.
$\w\in Q_{\Gamma}$ the sequences $\{\mathcal T \varphi_1(\w),\dots, \mathcal T\varphi_{\ell}(\w)\}$ form a frame of $J_V(\w) \subseteq \ell^2(\Gamma)$. Therefore, $\dim(J_V(\w))\le \ell,$ for a.e. $\w\in Q_{\Gamma}.$
	
In our particular situation there is another description of the range function $J_V(\w)$ associated to $V$.	
For each $\w\in Q_{\Gamma},$ define
	\[
	\theta_{\w}:= \{\gamma\in\Gamma\colon \chi_{\Omega}(\w+\gamma)\neq 0 \}, \mbox{ and } \ell_{\w}:=\#\theta_{\w}.
	\]
Write $\ell_{\w} =0$ if $\theta_\omega = \emptyset.$ Then, there exist $\gamma_1(\w), \dots, \gamma_{\ell_{\w}}(\w)\in\Gamma$ such that $w+\gamma_j(\w)\in\Omega,$ for all
	$j=1, \dots, \ell_{\w},$ which implies that $J_V(\w)\subseteq \ell^2(\{\delta_{\gamma_1(\w)}, \dots, \delta_{\gamma_{\ell_{\w}}(\w)}\}),$ for a.e.
	$\w\in Q_{\Gamma}.$
	Moreover, as in Corollary 2.8. of \cite{AAC15},
	$J_V(\w)=\ell^2(\{\delta_{\gamma_1(\w)}, \dots, \delta_{\gamma_{\ell_{\w}}(\w)}\}),$ for a.e. $\w\in Q_{\Gamma}.$
Thus, $\dim(J_V(\w))=\ell_{\w},$ which implies that $\ell_{\w}\le\ell,$ for a.e. $\w\in Q_{\Gamma},$ and therefore we obtain that
$$
	F_{\Omega,\Gamma}(\w)=\sum_{\gamma\in\Gamma} \chi_{\Omega}(\w+\gamma)\le\ell, \quad \mbox{ for a.e. }\w\in Q_{\Gamma}\,. 
$$
This shows that $(\Omega, \Gamma)$ is an $m$-subtiling pair for $\widehat G$ with $m \leq \ell.$ 
	
(2) Since $\Omega$ is bounded, by Proposition \ref{Prop1.1} there exists a bounded set $\Delta$ containing $\Omega$ which is an $\ell$-tile of $\widehat G$ by $\Gamma.$
Now using Theorem 4.1 of \cite{AAC15}, there exist $a_1, \dots , a_\ell \in G$ such that $ E_\Lambda (a_1, \dots , a_\ell)$
is a Riesz basis of $L^2(\Delta).$ As a consequence, $E_\Lambda (a_1, \dots , a_\ell)$ is a frame of $L^2(\Omega).$
	
\qed

\begin{Remark}
	Note that $\Omega$ does not need to be bounded: for example, $E_{\mathbb Z}(0)= \{e^{2\pi i k x} : k \in \mathbb Z\}$ is an orthonormal basis for $L^2(\Omega)$ for $\Omega = \bigcup_{n=0}^\infty n + (\frac{1}{2^{n+1}}, \frac{1}{2^{n}}]\subset \mathbb R$ and $\Omega$ is not bounded.
	However, for the proof of part (2) of Theorem \ref{Theorem1} we need  $\Omega$ to be bounded since the proof uses Proposition \ref{Prop1.1}.
	
\end{Remark}

\begin{Remark}
	Theorem \ref{Theorem1} for the case $\ell =1$ can be found in \cite{BHM2016}. In this case, the proof does not require making use of either the Paley-Wiener space of $\Omega$ or the range function associated to it as in the proof given above.
	
\end{Remark}

%
%

\begin{Remark}
In Part (1) of Theorem \ref{Theorem1} the inequality $m \leq \ell$ can be strict as the following example shows:
choose  $\Omega\subset \R^d$ such that $(\Omega , \mathbb Z^d)$ is an $\ell$-tiling pair for $\mathbb R^d$ and pick $a_1,\dots,a_{\ell}$ such that $E_{\Z^d}(a_1, \dots, a_\ell)$ is a Riesz basis
of $L^2(\Omega)$. Let $\Omega_0\subset \Omega$ be any subset of $\Omega$ such that $(\Omega_0 , \Z^d )$ is an  $(\ell$-1)-tiling pair of $\R^d$ (for example, remove from $\Omega$ a fundamental domain of $\Z^d$ in $\R^d).$
Then   $E_{\Z^d}(a_1, \dots, a_\ell)$ is a frame for $L^2(\Omega_0)$, and $(\Omega_0 , \Z^d )$ is not an $\ell$-subtiling pair for $\R^d.$
\end{Remark}

\

\section{Optimal frame bounds for sets of exponentials.}
\label{Sec-optimal}

The purpose of this section is to develop another condition guaranteeing when a set of exponentials of the form 
$$E_\Lambda(a_1, \dots, a_m) : =\{e_{a_j+\lambda}: j=1,2, \dots, m, \lambda \in \Lambda\}
$$
forms a frame for $L^2(\Omega)$, where $(\Omega, \Gamma)$ is an $\ell$-subtiling pair for $\widehat G$, as well as to find optimal frame bounds for this frame.

For the $\ell$-subtiling pair $(\Omega, \Gamma)$ of $\widehat G$, let $E$ be the set of measure zero in $Q_\Gamma$ such that $F_{\Omega, \Gamma} >  \ell$, and let $Q_0 := \{\w \in Q_\Gamma: F_{\Omega,\Gamma}(\w) = 0  \}.$
Let
$$
\widetilde{Q_\Gamma} := Q_\Gamma \setminus (Q_0 \cup E)\,.
$$
For each $\w \in \widetilde{Q_\Gamma} $ there exist $\ell_{\w}\leq \ell$ and $\ga_1(\w), \dots, \ga_{\ell_\w} (\w) \in \Gamma$ such that $\w + \ga_j(\w) \in \Omega$ for all $j=1, \dots, \ell_\w$ (see the proof of Theorem \ref{Theorem1}). Recall that
\begin{equation} \label{Eqellw}
\ell_{\w} := \#\{\gamma\in\Gamma\colon \chi_{\Omega}(\w+\gamma)\neq 0 \}\,. 
\end{equation}

Given $\p_1, \dots, \p_m \in PW_\Omega = \{f\in L^2(G): \widehat f \in L^2(\Omega)\}$, and $\w \in \widetilde Q_\Gamma $, consider the matrix
\begin{equation} \label{Eq3.1}
T_\omega = \left( \begin{array}{ccc}  \widehat \p_1(\w + \gamma_1(\w)) & \dots & \widehat \p_{m}(\w + \gamma_1(\w)) \\ \vdots &  &  \vdots \\ \widehat \p_1(\w + \gamma_{\ell_\w}(\w)) & \dots & \widehat \p_{m}(\w + \gamma_{\ell_\w}(\w))
\end{array}\right)\,
\end{equation}
of size $\ell_{\w}\times m.$
Assume that 
$$
\Phi_\Lambda : = \{t_\lambda \p_j: \lambda \in \Lambda, j=1, \dots, m\}
$$
is a frame for $S_\Lambda(\p_1, \cdots, \p_m)$.  By Proposition \ref{Prop2.3}, this is equivalent to having that for a.e. $\w\in Q_\Gamma$ the set
$$
\Phi_\w:= \{\mathcal T\p_j(\w):j=1, \dots, m\} \subset \ell^2(\Gamma)
$$
is a frame for $J(\w) = \overline {\mbox{span}} \{\mathcal T \p_1 (\w), \dots, \mathcal T \p_m (\w) \} \subset \ell^2(\Gamma)\,.$ 
Moreover, as in the proof of Theorem \ref{Theorem1}, for a. e. $\w \in Q_\Gamma$,
$J(\w) = \ell^2 (\{\delta_{\gamma_1(\w)}, \dots, \delta_{\gamma_{\ell_{\w}}(\w)}\})$ is a subspace of $\ell^2(\Gamma)$ of dimension $\ell_{\w}\,.$  (Notice that this implies $m \geq \ell$.) 

It is well known (see, for example, Proposition 3.18 in \cite{HKLW07}) that a frame in a finite dimensional Hilbert space is nothing but a generating set. Since the non-zero elements of $\mathcal T\p_j(\w)$ are precisely the $j$-th column of $T_\w$, $j=1, \dots, m\,,$ it follows that $\Phi_\Lambda$ is a frame for $S_\Lambda(\p_1, \cdots, \p_m)$ if and only if $\mbox{rank} \,( T_\w) = \ell_{\w}$ for a.e. $\w \in \widetilde Q_\Gamma.$

\

For $\w \in \widetilde{Q_\Gamma}$, let $\lambda_{min}(T_\w T_\w^*)$ and $\lambda_{max}(T_\w T_\w^*)$ respectively the minimal and maximal eigenvalues of $T_\w T_\w^*$. It is well known (see Proposition 3.27 in \cite{HKLW07}) that the optimal lower and upper frame bounds of $\Phi_\w$ are precisely $\lambda_{min}(T_\w T_\w^*)$ and $\lambda_{max}(T_\w T_\w^*)$ respectively. By Proposition \ref{Prop2.3} the optimal frame bounds for $\Phi_\Lambda$ are
\begin{equation} \label{Eq3.2}
A = |Q_\Gamma|\, \mbox{ess\,inf}_{\w \in \widetilde Q_\Gamma} \lambda_{min}(T_\w T_\w^*) \qquad \mbox{and} \qquad  B = |Q_\Gamma|\, \mbox{ess\,sup}_{\w \in \widetilde Q_\Gamma} \lambda_{max}(T_\w T_\w^*)\,.
\end{equation}
\

We have proved the following result:

\begin{proposition} \label{Pro3.1}
	With the notation and  definitions as above, the following are equivalent:
	 	
(i) The set $\Phi_\Lambda : = \{t_\lambda \p_j: \lambda \in \Lambda, j=1, \dots, m\}$ is a frame for $S_\Lambda(\p_1, \dots, \p_m)\,.$

(ii) The matrix $T_\w$ given in \eqref{Eq3.1} has rank $\ell_{\w}$ (see \eqref{Eqellw}) for a.e. $\w \in \widetilde Q_\Gamma.$

Moreover, in this situation, the optimal frame bounds $A$ and $B$ of $\Phi_\Lambda$ are given by \eqref{Eq3.2}.
\end{proposition}

Consider now the set of exponentials
\[
E_\Lambda(a_1,\dots,a_m):= \{e_{\lambda+a_j} \colon \lambda\in\Lambda, j=1,\dots, m\} 
\]
with $a_1, \dots, a_{m}\in G$. Let $\p \in L^2(G)$ given by $\widehat \p= \chi_\Omega.$ Consider
$$
\p_j := t_{-a_j} \p\,, \quad j=1, \dots, m.
$$
As in the proof of Theorem \ref{Theorem1}, $E_\Lambda (a_1,\dots,a_m)$ is a frame for $L^2(\Omega)$ with frame bounds $A$ and $B$ if and only if the set 
$$
\Phi_\Lambda : = \{t_\lambda \p_j: \lambda \in \Lambda, j=1, \dots, m\}
$$
is a frame for $PW_\Omega = S_\Lambda(\p_1, \cdots, \p_m)\,$ with the same frame bounds. 

For our particular situation, if $\w \in \widetilde{Q_\Gamma}$,
\begin{equation} \label{Eq3.3}
T_\omega = \left( \begin{array}{ccc}  e_{a_1}(\w+\gamma_1(\w)) & \dots & e_{a_m}(\w+\gamma_1(\w)) \\ \vdots &  &  \vdots \\ e_{a_1}(\w+\gamma_{\ell_\w}(\w)) & \dots & e_{a_m}(\w+\gamma_{\ell_\w}(\w))
\end{array}\right)\,.
\end{equation}
As in Theorem 2.9 of \cite{AAC15} the matrix $T_\w$, for $\w \in \widetilde Q_\Gamma$,   can be factored as
\begin{equation} \label{Eq3.4}
T_\w = E_\w U_\w : = \left( \begin{array}{ccc}  e_{a_1}(\gamma_1(\w)) & \dots & e_{a_m}(\gamma_1(\w))\\ \vdots &  &  \vdots \\ e_{a_1}(\gamma_{\ell_\w}(\w))& \dots & e_{a_m}(\gamma_{\ell_\w}(\w))
\end{array}\right) 
\left( \begin{array}{ccc}  e_{a_1}(\w)& \dots & 0 \\ \vdots &  &  \vdots \\ 0 & \dots & e_{a_m}(w)
\end{array}\right) 
\,.
\end{equation}

Since $U_\w$ is unitary and $T_\w T_\w^* = E_\w E_\w^*$, we have proved the following result:

\begin{proposition} \label{Prop3.2}
	With the notation and definitions as above, the following are equivalent:
	
(i) The set $E_\Lambda(a_1, \dots, a_m )$ is a frame for $L^2(\Omega)$.

(ii) The matrix $E_\w$ given in \eqref{Eq3.4} has rank $\ell_{\w}$ (see \eqref{Eqellw}) for a. e. $\w \in \widetilde{Q_\Gamma}.$

Moreover, in this situation, the optimal frame bounds $A$ and $B$ of $E_\Lambda(a_1, \dots, a_m )$  are given by
$$
A = |Q_\Gamma|\, \mbox{ess\,inf}_{\w \in \widetilde Q_\Gamma} \lambda_{min}(E_\w E_\w^*) \qquad \mbox{and} \qquad  B = |Q_\Gamma|\, \mbox{ess\,sup}_{\w \in \widetilde Q_\Gamma} \lambda_{max}(E_\w E_\w^*)\,.
$$

\end{proposition}

\

\begin{Remark}
	Proposition \ref{Prop3.2} can be found in \cite{AAC15} when $\Omega$ is an $\ell$-tile and ``frame''   is replaced by ``Riesz basis''. \end{Remark}

\

\begin{example}
In this example we work with the additive group $G=\mathbb R^d$  and the lattice $\Lambda =\mathbb Z^d.$ Recall that $\widehat{G} = \R^d$ and $\Gamma = \Z^d.$
Let $\Omega_0 \subset \Omega_1 \subset [0 , 1)^d$ be two measurable sets in $\R^d$ and let $\ga_0 \in \Z^d \ (\ga_0 \neq 0).$ Take
$$ \Omega = \Omega_1 \cup (\ga_{0} + \Omega_0)\,,$$
so that $(\Omega, \Z^d)$ is a 2-subtiling pair of $\R^d.$

For $a_1, a_2, \dots, a_m \in \R^d$ consider the set of exponentials
$$ E_{\Z^d} (a_1, \dots, a_m)= \{ e^{2\pi i \langle k+a_j , \,\bigcdot \,\rangle} : k \in \Z^d, j=1, \dots, m \}\,.$$ 
By factoring out $e^{2\pi i \langle a_1 , x \rangle}$ we can assume $a_1=0.$

According to Proposition \ref{Prop3.2}, to determine the values of $a_1=0, a_2, \dots, a_m$ for which the set $E_{\Z^d}(0, a_2, \dots, a_m)$ is a frame for $L^2(\Omega)$, we need to compute the ranks of the matrices $E_\w$ given in \eqref{Eq3.4}.

For $\w \in \Omega_1 \setminus \Omega_0,\ \ell_{\w}=1$, $E_w = (1, 1, \dots , 1)$, and $\mbox{rank}\, (E_\w) = 1 = \ell_{\w}.$ For $\w \in \Omega_0,\ \ell_{\w}=2,$ and

\begin{eqnarray} \label{Eq3.5}
E_\w = \left( \begin{array}{cccc}  1 & 1 & \dots & 1 \\ 1 & e^{2\pi i \langle a_2 , \ga_{0}\rangle } & \dots &  e^{2\pi i \langle a_m , \ga_{0}\rangle } 
\end{array}\right) \,.
\end{eqnarray}
Let $H := \bigcup_{k\in \Z}\{x \in \R^d: \langle x, \gamma_0 \rangle = k  \}\,,$ that is a countable union of hyperplanes in $\R^d$ perpendicular to the vector $\ga_{0}.$
The rank of the matrix given in \eqref{Eq3.5} is 2 when at least one of the $a_j$ does not belong to $H$. In this case,  $E_{\Z^d}(0, a_2, \dots, a_m)$ is a frame for $L^2(\Omega)$ as an application of Proposition \ref{Prop3.2}.

We now compute the optimal frame bounds. For $\w \in \Omega_1 \setminus \Omega_0\,, E_\w E_\w^* =(m),$ so that $ \lambda_{\mbox{min}}(E_\w E_\w^*) = \lambda_{\mbox{max}}(E_\w E_\w^*) = m\,.$ 
For $\w \in \Omega_0\,,$ 
$$
E_\w E_\w^*= \left( \begin{array}{cc}  m & 1 + \sum_{j=2}^m e^{-2 \pi i \langle a_j , \ga_{0}\rangle} \\ 1 + \sum_{j=2}^m e^{2 \pi i \langle a_j , \ga_{0}\rangle}  & m 
\end{array}\right) \,.
$$
The eigenvalues of this matrix are
$$
\lambda = m \pm \Big| 1 + \sum_{j=2}^m e^{2 \pi i \langle a_j , \ga_{0}\rangle}\Big|\,.
$$
Therefore, the optimal lower and upper frame bounds of  $E_{\Z^d}(0, a_2, \dots, a_m)$ in $L^2(\Omega)$ are 
$$
A = m - \Big| 1 + \sum_{j=2}^m e^{2 \pi i \langle a_j , \ga_{0}\rangle}\Big|\, \quad \mbox{and}\quad B= m + \Big| 1 + \sum_{j=2}^m e^{2 \pi i \langle a_j , \ga_{0}\rangle}\Big|
$$
when $a_j \notin H$ for some $j\in \{2, \dots, m \}.$ Observe that the frame $E_{\Z^d}(0, a_2, \dots, a_m)$ in $L^2(\Omega)$ is tight (with tight frame bound $m$) if and only if $\displaystyle 1 + \sum_{j=2}^m e^{2 \pi i \langle a_j , \ga_{0}\rangle}=0$. This occurs, for example, if the complex numbers $\{ 1, e^{2 \pi i \langle a_2 , \ga_{0}\rangle}, \dots , e^{2 \pi i \langle a_m , \ga_{0}\rangle}  \}$ are the vertices of a regular $m$-gon inscribed in the unit circle.
\end{example}

\bibliographystyle{plain}

\begin{thebibliography}{20}
	
\bibitem{AAC15} E. Agora, J. Antezana, and C. Cabrelli, {\it Muti-tiling sets, Riesz bases, and sampling near the critical density in LCA groups}. Advances in Math., 285 (2015), 454--477.

\bibitem{BHM2016}
D.~Barbieri, E.~Hern\'andez and A.~Mayeli.
\newblock {\em Lattice sub-tilings and frames  in {LCA} groups.} C. R. Acad. Sci. Paris, Ser. 1, 356 (2), (2017), 193-199.

\bibitem{CP10}
C.~Cabrelli and V.~Paternostro.
\newblock {\em Shift-invariant spaces on {LCA} groups.}
\newblock J. Funct. Anal., 258(6), (2010), 2034--2059.


\bibitem{FG68}
J.~Feldman and F.P. ~Greenleaf.
\newblock {\em Existence of Borel transversals in groups.}
\newblock Pacific J. Math. 25 (1968) 455–-461.
	
\bibitem{Fue74} B. Fuglede, {\it Commuting self-adjoint partial differential operators and a group theoretic problem}. J. Funct. Anal. 16 (1974), 101--121.

\bibitem{GL14} S. Grepstad, N. Lev, {\it Multi-tiling and Riesz basis}. Advances in Math., 252 (15), (2014), 1--6.

\bibitem{HKLW07} D. Han, K. Kornelson, D. Larson, E. Weber, {\it Frames for undergraduates}. AMS, Student Mathematical Library, Vol. 40, (2007).

\bibitem{HR79} E. Hewitt, K. A. Ross, {\it Abstract harmonic analysis. Vol. I: Structure of topological groups, integration theory, group representations.} Springer, 2nd Ed. (1979).

\bibitem{Kol15} M. Kolountzakis, {\it Multiple lattice tiles and Riesz bases of exponentials}.  Proc. Amer. Math. Soc. 143 (2015), 741--747.
	
\bibitem{Pedersen} S. Pedersen, {\it Spectral Theory of Commuting Self-Adjoint Partial Differential Operators}. Journal of Functional Analysis 73 (1987), 122--134 .

\bibitem{Young} R. M. Young, {\it Introduction to nonharmonic Fourier series}. Academic Press, (1980).



\end{thebibliography}

\vskip 1truemm

\end{document}